\documentclass{icm2010}
\usepackage{verbatim,hyperref,mathrsfs}
\title[Embeddings of the Heisenberg group]{$L_1$ embeddings of the Heisenberg group and fast estimation of graph isoperimetry}
\author[Assaf Naor]{Assaf Naor\thanks{Research supported in part by NSF grants CCF-0635078 and CCF-0832795, BSF grant 2006009, and the Packard Foundation.}
}
\contact[naor@cims.nyu.edu]{New York University, Courant Institute of Mathematical Sciences, 251 Mercer Street, New York, NY 10012, USA}

\newtheorem{theorem}{Theorem}[section]

\newtheorem{lemma}[theorem]{Lemma}
\newtheorem{cor}[theorem]{Corollary}
\theoremstyle{definition}
\newtheorem{definition}[theorem]{Definition}
\newcommand{\R}{\mathbb{R}}
\newcommand{\eqdef}{\stackrel{\mathrm{def}}{=}}
\newcommand{\1}{\mathbf{1}}
\newcommand{\M}{\mathscr{M}}
\newcommand{\N}{\mathbb{N}}
\newcommand{\C}{\mathbb{C}}
\renewcommand{\H}{\mathbb{H}}
\newcommand{\Z}{\mathbb{Z}}
\renewcommand{\le}{\leqslant}
\renewcommand{\ge}{\geqslant}
\renewcommand{\leq}{\leqslant}

\newcommand{\Y}{\mathscr{Y}}
\newcommand{\e}{\varepsilon}
\newcommand{\f}{\varphi}
\newtheorem{rem}[theorem]{Remark}
\renewcommand{\L}{\mathscr{L}}
\newcommand{\Cut}{\mathrm{Cut}}
\begin{document}

\begin{abstract} We survey connections between the theory of
bi-Lipschitz embeddings and the Sparsest Cut Problem in
combinatorial optimization. The story of the Sparsest Cut Problem is
a striking example of the deep interplay between analysis, geometry,
and probability on the one hand, and computational issues in
discrete mathematics on the other.
 We
explain how the key ideas evolved over the past 20 years,
emphasizing the interactions with Banach space theory, geometric
measure theory, and geometric group theory.
As an important illustrative example, we shall examine recently
established connections to the the structure
of the Heisenberg group, and the incompatibility of its
Carnot-Carath\'eodory geometry with the geometry of the Lebesgue space $L_1$.
\end{abstract}



\maketitle



\section{Introduction}

Among the common definitions of the Heisenberg group $\H$, it will
be convenient for us to work here with $\H$ modeled as $\R^3$,
equipped with the group product $(a,b,c)\cdot (a',b',c')\eqdef
(a+a',b+b',c+c'+ab'-ba')$. The integer lattice $\Z^3$ is then a
discrete cocompact subgroup of $\H$, denoted by $\H(\Z)$, which is
generated by the finite symmetric set $\{(\pm1,0,0),(0,\pm1,0),(0,0,\pm1)\}$.
The word metric on $\H(\Z)$ induced by this generating set will be
denoted by $d_W$.

As noted by Semmes~\cite{Sem96}, a  differentiability
result of Pansu~\cite{Pan89} implies that the metric space
$(\H(\Z),d_W)$ does not admit a bi-Lipschitz embedding into $\R^n$
for any $n\in \N$. This was extended by Pauls~\cite{Pau01} to
bi-Lipschitz non-embeddability results of $(\H(\Z),d_W)$ into metric
spaces with either lower or upper curvature bounds in the sense of
Alexandrov. In~\cite{LN06,CK06-RNP} it was observed that Pansu's
differentiability argument extends to Banach space targets with the
Radon-Nikod\'ym property (see~\cite[Ch.\ 5]{benlin}), and hence
$\H(\Z)$ does not admit a bi-Lipschitz embedding into, say, a Banach
space which is either reflexive or is a separable dual; in
particular $\H(\Z)$ does not admit a bi-Lipschitz embedding into any
$L_p(\mu)$ space, $1<p<\infty$, or into the sequence space $\ell_1$.

The embeddability of $\H(\Z)$ into the function space $L_1(\mu)$,
when $\mu$ is non-atomic, turned out to be much harder to settle.
This question is of particular importance since it is well
understood that for $\mu$ non-atomic, $L_1(\mu)$ is a space for
which the differentiability results quoted above manifestly break
down. Nevertheless, Cheeger and Kleiner~\cite{CK-Cras,ckbv}
introduced a novel notion of differentiability for which they could
prove a differentiability theorem for Lipschitz maps from the
Heisenberg group to $L_1(\mu)$, thus establishing that $\H(\Z)$ does
not admit a bi-Lipschitz embedding into any $L_1(\mu)$ space.

Another motivation for the $L_1(\mu)$ embeddability question for
$\H(\Z)$ originates from~\cite{LN06}, where it was established that
it is connected to the Sparsest Cut Problem in the field of
combinatorial optimization. For this application it was of
importance to obtain quantitative estimates in the $L_1(\mu)$
non-embeddability results for $\H(\Z)$. It turns out that
establishing such estimates is quite subtle, as they require
overcoming finitary issues that do not arise in the infinite setting
of~\cite{ckbv,ckmetmon}. The following two theorems were proved
in~\cite{ckn,CKN09}. Both theorems follow painlessly from a more
general theorem that is stated and discussed in
Section~\ref{sec:compression}.
\begin{theorem}\label{thm:ckn}
There exists a universal constant $c>0$ such that any
embedding into $L_1(\mu)$ of the restriction of the word metric
$d_W$ to the $n\times n\times n$ grid $\{1,\ldots,n\}^3$ incurs
distortion $\gtrsim (\log n)^c$.
\end{theorem}
Following Gromov~\cite{Gro93}, the compression rate of $f:\H(\Z)\to
L_1(\mu)$, denoted $\omega_f(\cdot)$, is defined as the largest
non-decreasing function such that for all $x,y\in \H(\Z)$ we have
$\|f(x)-f(y)\|_1\ge \omega_f(d_W(x,y))$ (see~\cite{ADS06} for more
information on this topic).
\begin{theorem}\label{thm:rate} There exists a universal constant
$c>0$ such that for every function $f:\H(\Z)\to L_1(\mu)$ which is
$1$-Lipschitz with respect to the word metric $d_W$, we have
$\omega_f(t)\lesssim t/(\log t)^c$ for all $t\ge 2$.
\end{theorem}
 Evaluating the supremum of those $c>0$ for which
Theorem~\ref{thm:ckn} holds true remains an important open question,
with geometric significance as well as importance to theoretical
computer science. Conceivably we could get $c$ in
Theorem~\ref{thm:ckn} to be arbitrarily close to $\frac12$, which
would be sharp since the results of~\cite{Ass83,Rao99} imply (see
the explanation in~\cite{GKL03}) that the metric space
$\left(\{1,\ldots,n\}^3,d_W\right)$ embeds into $\ell_1$ with
distortion $\lesssim \sqrt{\log n}$. Similarly, we do not know the
best possible $c$ in Theorem~\ref{thm:rate}; $\frac12$ is again the
limit here since it was shown in~\cite{Tess08} that there exists a
$1$-Lipschitz mapping $f:\H(\Z)\to \ell_1$ for which
$\omega_f(t)\gtrsim t/(\sqrt{\log t}\cdot \log \log t)$.

The purpose of this article is to describe the above non-embeddability
results for the Heisenberg group. Since one of the motivations
for these investigations is the application to the Sparsest Cut
Problem, we also include here a detailed discussion of this problem
from theoretical computer science, and its deep connections to
 metric geometry. Our goal is to present the ideas in
a way that is accessible to mathematicians who do not necessarily have
background in computer science.

\bigskip
\noindent {\bf Acknowledgements.} I am grateful to the
following people for helpful comments and suggestions on earlier versions of this
manuscript: Tim Austin, Keith Ball, Subhash Khot, Bruce Kleiner, Russ Lyons,
Manor Mendel, Gideon Schechtman, Lior Silberman.


\section{Embeddings}\label{SEC:EMBED}

A metric space $(\M,d_\M)$ is said to embed with distortion $D\ge 1$
into a metric space $(\Y,d_Y)$ if there exists a mapping $f:\M\to
\Y$, and a scaling factor $s>0$, such that for all $x,y\in \M$ we
have $sd_\M(x,y)\le d_\Y(f(x),f(y))\le Dsd_\M(x,y)$. The infimum
over those $D\ge 1$ for which $(\M,d_\M)$ embeds with distortion $D$
into $(\Y,d_Y)$ is denoted by $c_\Y(\M)$. If $(\M,d_\M)$ does not
admit a bi-Lipschitz embedding into $(\Y,d_Y)$, we will write
$c_\Y(\M)=\infty$.

Throughout this paper, for $p\ge 1$, the space $L_p$ will stand for
$L_p([0,1],\lambda)$, where $\lambda$ is Lebesgue measure. The
spaces $\ell_p$ and $\ell_p^n$ will stand for the space of
$p$-summable infinite sequences, and $\R^n$ equipped with the
$\ell_p$ norm, respectively. Much of this paper will deal with
bi-Lipschitz embeddings of {\em finite} metric spaces into $L_p$.
Since every $n$-point subset of an $L_p(\Omega,\mu)$ space embeds
isometrically into $\ell_p^{n(n-1)/2}$ (see the discussion
in~\cite{Bal90}), when it comes to embeddings of finite metric
spaces, the distinction between different $L_p(\Omega,\mu)$ spaces
is irrelevant. Nevertheless, later, in the study of the
embeddability of the Heisenberg group, we will need to distinguish
between sequence spaces and function spaces.

For $p\ge 1$ we will use the shorter notation $c_p(\M)=
c_{L_p}(\M)$. The parameter $c_2(\M)$ is known as the Euclidean
distortion of $\M$. Dvoretzky's theorem says that if $\Y$ is an
infinite dimensional Banach space then $c_\Y(\ell_2^n)=1$ for all
$n\in \N$. Thus, for every finite metric space $\M$ and every
infinite dimensional Banach space $\Y$, we have $c_2(\M)\ge
c_\Y(\M)$.

The following famous theorem of Bourgain~\cite{Bou85} will play a
key role in what follows:
\begin{theorem}[Bourgain's embedding theorem~\cite{Bou85}]\label{thm:bourgain}
For every $n$-point metric space $(\M,d_\M)$, we have
\begin{equation}\label{eq:bourgain estimate}
c_2(\M)\lesssim \log n.
\end{equation}
\end{theorem}
Bourgain proved in~\cite{Bou85} that the estimate~\eqref{eq:bourgain
estimate} is sharp up to an iterated logarithm factor, i.e., that
there exist arbitrarily large $n$-point metric spaces $\M_n$ for
which $c_2(\M_n)\gtrsim \frac{\log n}{\log \log n}$. The $\log\log n$ term was removed in the important paper~\cite{LLR95} of
Linial, London and Rabinovich, who showed that the shortest path
metric on bounded degree $n$-vertex expander graphs has Euclidean
distortion $\gtrsim \log n$.

If one is interested only in embeddings into infinite dimensional
Banach spaces, then Theorem~\ref{thm:bourgain} is stated in the
strongest possible form: as noted above, it implies that for every
infinite dimensional Banach space $\Y$, we have $c_\Y(\M)\lesssim
\log n$. Below, we will actually use Theorem~\ref{thm:bourgain} for
embeddings into $L_1$, i.e., we will use the fact that
$c_1(\M)\lesssim \log n$. The expander based lower bound of Linial,
London and Rabinovich~\cite{LLR95} extends to embeddings into $L_1$
as well, i.e., even this weaker form of Bourgain's embedding theorem
is asymptotically sharp. We refer to~\cite[Ch.\  15]{Mat02} for a
comprehensive discussion of these issues, as well as a nice
presentation of the proof of Bourgain's embedding theorem.

\section{$L_1$ as a metric space}\label{sec:L1}

Let $(\Omega,\mu)$ be a measure space. Define a mapping
$T:L_1(\Omega,\mu)\to L_\infty(\Omega\times \R,\mu\times \lambda)$,
where $\lambda$ is Lebesgue measure, by:
\begin{equation*}\label{eq:def T}
T(f)(\omega,x)\eqdef \left\{
\begin{array}{ll}1& 0< x\le f(\omega),\\
-1 & f(\omega)<x<0,\\
0& \mathrm{otherwise}.\end{array}\right.
\end{equation*}
For all $f,g\in L_1(\Omega,\mu)$ we have:
\begin{equation*}\label{eq:difference}
\Big|T(f)(\omega,x)-T(g)(\omega,x)\Big|= \left\{
\begin{array}{ll}1& g(\omega)< x\le f(\omega)\ \mathrm{or}\ f(\omega)< x\le g(\omega),\\
0& \mathrm{otherwise}.\end{array}\right.
\end{equation*}
Thus, for all $p>0$ we have,
\begin{multline}\label{eq:lp}
\left\|T(f)-T(g)\right\|_{L_p(\Omega\times \R,\mu\times \lambda)}^p
=\int_{\Omega} \left(\int_{(g(\omega),f(\omega)]\sqcup
(f(\omega),g(\omega)]} d\lambda\right)d\mu(\omega)\\=\int_\Omega
|f(\omega)-g(\omega)|d\mu(\omega)=\|f-g\|_{L_1(\Omega,\mu)}.
\end{multline}
Specializing~\eqref{eq:lp} to $p=2$, we see that:
\begin{equation*}\label{eq:NEG}
\left\|T(f)-T(g)\right\|_{L_2(\Omega\times \R,\mu\times \lambda)}=\sqrt{\|f-g\|_{L_1(\Omega,\mu)}}.
\end{equation*}
\begin{cor}\label{cor:NEG}
The metric space $\left(L_1(\Omega,\mu),\|f-g\|_{L_1(\Omega,\mu)}^{1/2}\right)$ admits an isometric embedding into Hilbert space.
\end{cor}

Another useful corollary is obtained when~\eqref{eq:lp} is
specialized to the case $p=1$. Take an arbitrary finite subset
$X\subseteq L_1(\Omega,\mu)$. For every $(\omega,x)\in \Omega\times
\R$ consider the set $S(\omega,x)= \{f\in X:\ x\le
f(\omega)\}\subseteq X$. For every $S\subseteq X$ we can define a
measurable subset $E_S= \{(\omega,x)\in \Omega\times \R:\
S(\omega,x)=S\}\subseteq \Omega\times \R$. By the definition of $T$,
for every $f,g\in X$ we have
\begin{eqnarray*}\label{proof:cut}
\|f-g\|_{L_1(\Omega,\mu)}&\stackrel{\eqref{eq:lp}}{=}
&\left\|T(f)-T(g)\right\|_{L_1(\Omega\times \R,\mu\times \lambda)}\nonumber
\\\nonumber&=&\int_{\Omega\times \R} \Big|\1_{S(w,x)}(f)-\1_{S(w,x)}(g)\Big|\,
d(\mu\times \lambda)(\omega,x)\\&=&\sum_{S\subseteq X} (\mu\times \lambda)(E_S)\Big|\1_{S}(f)-\1_{S}(g)\Big|,
\end{eqnarray*}
where here, and in what follows, ${\bf 1}_S(\cdot)$ is the
characteristic function of $S$. Writing 
$\beta_S=(\mu\times \lambda)(E_S)$, we have the following important
corollary:
\begin{cor}\label{cor:cut}
Let $X\subseteq L_1(\Omega,\mu)$ be a finite subset of $L_1(\Omega,\mu)$. Then there exist nonnegative numbers $\{\beta_S\}_{S\subseteq X}\subseteq [0,\infty)$ such that for all $f,g\in X$ we have:
\begin{equation}\label{eq:superposition}
\|f-g\|_{L_1(\Omega,\mu)}=\sum_{S\subseteq X} \beta_S\Big|\1_{S}(f)-\1_{S}(g)\Big|.
\end{equation}
\end{cor}

A metric space $(\M,d_\M)$ is said to be of {\em negative type} if
the metric space $\left(\M,d_\M^{1/2}\right)$ admits an isometric
embedding into Hilbert space. Such metrics will play a crucial role
in the ensuing discussion. This terminology (see e.g.,
\cite{dezalaur}) is due to a classical theorem of
Schoenberg~\cite{schoenberg}, which asserts that $(\M,d_\M)$ is of
negative type if and only if for every $n\in \N$ and every
$x_1,\ldots,x_n\in X$, the matrix $(d_\M(x_i,x_j))_{i,j=1}^n$ is
negative semidefinite on the orthogonal complement of the main
diagonal in $\C^n$, i.e., for all $\zeta_1,\ldots,\zeta_n\in \C$
with $\sum_{j=1}^n\zeta_j=0$ we have $\sum_{i=1}^n\sum_{j=1}^n
\zeta_i\overline{\zeta_j}d_\M(x_i,x_j)\le 0$.
Corollary~\eqref{cor:NEG} can be restated as saying that
$L_1(\Omega,\mu)$ is a metric space of negative type.

Corollary~\eqref{cor:cut} is often called the {\em cut cone
representation of $L_1$ metrics}. To explain this terminology,
consider the set $\mathscr C\subseteq \R^{n^2}$ of all $n\times n$
real matrices $A=(a_{ij})$ such that there is a measure space
$(\Omega,\mu)$ and $f_1,\ldots,f_n\in L_1(\Omega,\mu)$ with
$a_{ij}=\|f_i-f_j\|_{L_1(\Omega,\mu)}$ for all $i,j\in
\{1,\ldots,n\}$. If $f_1,\ldots,f_n\in L_1(\Omega_1,\mu_1)$ and
$g_1,\ldots,g_n\in L_1(\Omega_2,\mu_2)$ then for all $c_1,c_2\ge 0$
and $i,j\in \{1,\ldots,n\}$ we have
$$c_1\|f_i-f_j\|_{L_1(\Omega_1,\mu_1)}+c_2\|f_i-f_j\|_{L_1(\Omega_2,\mu_2)}
=\|h_i-h_j\|_{L_1(\Omega_1\sqcup\Omega_2,\mu_1\sqcup\mu_2)},$$ where
$h_1,\ldots,h_n$ are functions defined on the disjoint union
$\Omega_1\sqcup\Omega_2$ as follows:
$h_i(\omega)=c_1f_i(\omega)\1_{\Omega_1}(\omega)+c_2g_i(\omega)\1_{\Omega_2}(\omega)$.
This observation shows that $\mathscr C$ is a cone (of dimension
$n(n-1)/2$). Identity~\eqref{eq:superposition} says that the cone
$\mathscr C$ is generated by the rays induced by cut semimetrics,
i.e., by matrices of the form $a_{ij}=|\1_S(i)-\1_S(j)|$ for some
$S\subseteq \{1,\ldots,n\}$. It is not difficult to see that these
rays are actually the extreme rays of the cone $\mathscr C$.
Carath\'eodory's theorem (for cones) says that we can choose the
coefficients $\{\beta_S\}_{S\subseteq X}$
in~\eqref{eq:superposition} so that only $n(n-1)/2$ of them are
non-zero.

\section{The Sparsest Cut Problem}

Given $n\in \N$ and two symmetric functions
$C,D:\{1,\ldots,n\}\times\{1,\ldots,n\}\to [0,\infty)$ (called
capacities and demands, respectively), and a subset $\emptyset\neq S
\subsetneq \{1,\ldots,n\}$, write
\begin{equation}\label{eq:defPhi}
\Phi(S) \eqdef \frac{\sum_{i=1}^n\sum_{j=1}^n C(i,j) \cdot |{\bf
1}_S(i) - {\bf 1}_S(j)|}
               {\sum_{i=1}^n\sum_{j=1}^n D(i,j) \cdot |{\bf 1}_S(i) - {\bf 1}_S(j)|}.
\end{equation}
 The value
\begin{equation}\label{eq:defPhi*}
\Phi^*(C,D) \eqdef \min_{\emptyset\neq S \subsetneq \{1,\ldots,n\}} \Phi(S)\end{equation}
 is
the minimum over all cuts (two-part partitions) of $\{1,\ldots,n\}$ of the
ratio between the total capacity crossing the boundary of the cut
and the total demand crossing the boundary of the cut.

Finding in polynomial time a cut for which $\Phi^*(C,D)$ is attained
up to a definite multiplicative constant is called the Sparsest Cut
problem. This problem is used as a subroutine in many approximation
 algorithms for NP-hard problems; see the survey
 articles~\cite{Shmoys95,Chawla08}, as well as~\cite{LR99,AKRR90} and
 the references in~\cite{ARV04,ALN08} for some of the vast literature on this topic. Computing
$\Phi^*(C,D)$ exactly has been long known to be NP-hard~\cite{SM90}.
More recently, it was shown in~\cite{CK07} that there exists
$\e_0>0$ such that it is NP-hard to approximate $\Phi^*(C,D)$ to
within a factor smaller than $1+\e_0$. In~\cite{KV04,CKKRS06} it was
shown that it is Unique Games hard to approximate $\Phi^*(C,D)$ to
within any constant factor (see~\cite{Khot02,Khot10} for more information
on the Unique Games Conjecture; we will return to this issue in
Section~\ref{sec:unique}).

It is customary in the literature to highlight the support of the capacities function $C$:
this allows us to introduce a particulary important special case of the Sparsest Cut Problem. Thus, a different way to formulate the above setup is via an $n$-vertex graph $G=(V,E)$,
with a positive weight (called a capacity) $C(e)$ associated to each edge
$e\in E$, and a nonnegative weight (called a demand) $D(u,v)$ associated
to each pair of vertices $u,v\in V$. The goal is to evaluate in polynomial time
(and in particular, while examining only a negligible fraction of the subsets of $V$) the quantity:
\begin{equation*}\label{eq:def Phi}
\Phi^*(C,D)=\min_{\emptyset \neq S\subsetneq V}
\frac{\sum_{uv\in E}C(uv)\left|\1_S(u)-\1_S(v)\right|}{\sum_{u,v\in V}D(u,v)\left|\1_S(u)-\1_S(v)\right|}\,  .
\end{equation*}
To get a feeling
for the meaning of $\Phi^*$, consider the case $C(e)=D(u,v)=1$ for all $e\in E$ and $u,v\in V$. This is
an important instance of the Sparsest Cut problem which is called ``Sparsest Cut with Uniform Demands".
In this case $\Phi^*$ becomes:
\begin{equation*}\label{eq:def uniform}
\Phi^*= \min_{\emptyset \neq S\subsetneq V}\frac{\#\{\mathrm{edges\ joining}\ S\ \mathrm{and}\
V\setminus S\}}{|S|\cdot |V\setminus S|}\, .
\end{equation*}
Thus, in the case of uniform demands, the Sparsest Cut problem
essentially amounts to solving efficiently the combinatorial
isoperimetric problem on $G$: determining the subset of the graph
whose ratio of edge boundary to its size is as small as possible.

In the literature it is also customary to emphasize the size of the
support of the demand function $D$, i.e., to state bounds in terms
of the number $k$ of pairs $\{i,j\}\subseteq \{1,\ldots,n\}$ for
which $D(i,j)>0$. For the sake of simplicity of exposition, we will
not adopt this convention here, and state all of our bounds in terms
of $n$ rather than the number of positive demand pairs $k$. We refer
to the relevant references for the simple modifications that are
required to obtain bounds in terms of $k$ alone.

>From now on, the  Sparsest Cut problem will be understood to be with
general capacities and demands; when discussing the special case of
uniform demands we will say so explicitly. In applications, general
capacities and demands are used to tune the notion of ``interface"
between $S$ and $V\setminus S$ to a wide variety of combinatorial
optimization problems, which is
 one of the reasons why the Sparsest Cut problem is so versatile in the field of approximation algorithms.

\subsection{Reformulation as an optimization problem over $L_1$}
Although the Sparsest Cut Problem clearly has geometric flavor as a
discrete isoperimetric problem, the following key reformulation of
it, due to~\cite{AvisDeza,LLR95}, explicitly relates it to the
geometry of $L_1$.

\begin{lemma}\label{lem:l1 enters}
Given symmetric $C,D:\{1,\ldots,n\}\times\{1,\ldots,n\}\to
[0,\infty)$, we have:
\begin{equation}\label{reformulate Phi*}
\Phi^*(C,D) = \min_{f_1,\ldots,f_n\in
L_1}\frac{\sum_{i=1}^n\sum_{j=1}^n C(i,j) \|f_i-f_j\|_1}
               {\sum_{i=1}^n\sum_{j=1}^n D(i,j)\|f_i-f_j\|_1}.
\end{equation}
\end{lemma}
\begin{proof}
Let $\phi$ denote the right hand side of~\eqref{reformulate Phi*},
and write $\Phi^*=\Phi^*(C,D)$. Given a subset $S\subseteq
\{1,\ldots, n\}$, by considering $f_i=\1_S(i)\in \{0,1\}\subseteq
L_1$ we see that that $\phi\le \Phi^*$. In the reverse direction, if
$X=\{f_1,\ldots,f_n\}\subseteq L_1$ then let
$\{\beta_S\}_{S\subseteq X}$ be the non-negative weights from
Corollary~\ref{cor:cut}. For $S\subseteq X$ define a subset of
$\{1,\ldots,n\}$ by $S'= \{i\in\{1,\ldots,n\}:\ f_i\in S\}$. It
follows from the definition of $\Phi^*$ that for all $S\subseteq X$
we have,
\begin{multline}\label{eq;use def phi*}
\sum_{i=1}^n\sum_{j=1}^n
C(i,j)|\1_S(f_i)-\1_{S}(f_j)|\stackrel{\eqref{eq:defPhi}}=\Phi(S')\sum_{i=1}^n\sum_{j=1}^n
D(i,j)|\1_S(f_i)-\1_{S}(f_j)|\\ \stackrel{\eqref{eq:defPhi*}}{\ge}
\Phi^*\sum_{i=1}^n\sum_{j=1}^n D(i,j)|\1_S(f_i)-\1_{S}(f_j)|.
\end{multline}
Thus
\begin{multline*}\label{eq:use cut}
\sum_{i=1}^n\sum_{j=1}^n
C(i,j)\|f_i-f_j\|_1\stackrel{\eqref{eq:superposition}}{=}\sum_{S\subseteq
X}\beta_S \sum_{i=1}^n\sum_{j=1}^n
C(i,j)|\1_S(f_i)-\1_{S}(f_j)|\\\stackrel{\eqref{eq;use def
phi*}}{\ge} \Phi^*\sum_{S\subseteq X}\beta_S\sum_{i=1}^n\sum_{j=1}^n
D(i,j)|\1_S(f_i)-\1_{S}(f_j)|\stackrel{\eqref{eq:superposition}}{=}\sum_{i=1}^n\sum_{j=1}^n
D(i,j)\|f_i-f_j\|_1.
\end{multline*}
It follows 
that $\phi\ge \Phi^*$, as
required.
\end{proof}

\subsection{The linear program}\label{sec:linear}

Lemma~\ref{lem:l1 enters} is a reformulation of the Sparsest Cut
Problems in terms of a continuous optimization problem on the space
$L_1$. Being a reformulation, it shows in particular that solving
$L_1$ optimization problems such as the right hand side
of~\eqref{reformulate Phi*} is NP-hard.

In the beautiful paper~\cite{LR99} of Leighton and Rao it was shown
that there exists a polynomial time algorithm that, given an
$n$-vertex graph $G=(V,E)$, computes a number which is guaranteed to
be within a factor of $\lesssim \log n$ of the uniform Sparsest Cut
value~\eqref{eq:def uniform}. The Leighton-Rao algorithm uses
combinatorial ideas which do not apply to Sparsest Cut with general
demands. A breakthrough result, due to
Linial-London-Rabinovich~\cite{LLR95} and Aumann-Rabani~\cite{AR98},
introduced embedding methods to this field, yielding a polynomial
time algorithm which computes $\Phi^*(C,D)$ up to a factor $\lesssim
\log n$ for all $C,D:\{1,\ldots,n\}\times\{1,\ldots,n\}\to
[0,\infty)$.

The key idea of~\cite{LLR95,AR98} is based on replacing the finite
subset $\{f_1,\ldots,f_n\}$ of $L_1$ in~\eqref{reformulate Phi*} by
an {\em arbitrary} semimetric on $\{1,\ldots,n\}$. Specifically, by
homogeneity we can always assume that the denominator
in~\eqref{reformulate Phi*} equals $1$, in which case
Lemma~\ref{lem:l1 enters} says that $\Phi^*(C,D)$ equals the minimum
of $\sum_{i=1}^n\sum_{j=1}^nC(i,j)d_{ij}$, given that
$\sum_{i=1}^n\sum_{j=1}^n D(i,j)d_{ij}=1$ and there exist
$f_1,\ldots,f_n\in L_1$ for which $d_{ij}=\|f_i-f_j\|_1$ for all
$i,j\in \{1,\ldots,n\}$. We can now ignore the fact that $d_{ij}$
was a semimetric that came from a subset of $L_1$, i.e., we can
define $M^*(C,D)$ to be the minimum of
$\sum_{i=1}^n\sum_{j=1}^nC(i,j)d_{ij}$, given that
$\sum_{i=1}^n\sum_{j=1}^n D(i,j)d_{ij}=1$, $d_{ii}=0$, $d_{ij}\ge 0$, $d_{ij}=d_{ji}$ for
all $i,j\in \{1,\ldots,n\}$ ($n(n-1)/2$ symmetry constraints) and
$d_{ij}\le d_{ik}+d_{kj}$ for all $i,j,k\in \{1,\ldots,n\}$ ($\le
n^3$ triangle inequality constraints).

Clearly $M^*(C,D)\le \Phi^*(C,D)$, since we are minimizing over all
semimetrics rather than just those arising from subsets of $L_1$.
Moreover, $M^*(C,D)$ can be computed in polynomial time up to
arbitrarily good precision~\cite{GLS93}, since it is a linear
program (minimizing a linear functional in the variables $(d_{ij})$
subject to polynomially many linear constraints).

The linear program produces a semimetric $d_{ij}^*$ on
$\{1,\ldots,n\}$ which satisfies
$M^*(C,D)=\sum_{i=1}^n\sum_{j=1}^nC(i,j)d^*_{ij}$ and
$\sum_{i=1}^n\sum_{j=1}^n D(i,j)d^*_{ij}=1$ (ignoring arbitrarily
small errors). By Lemma~\ref{lem:l1 enters} we need to somehow
relate this semimetric to $L_1$. It is at this juncture that we see
the power of Bourgain's embedding theorem~\ref{thm:bourgain}: the
constraints of the linear program only provide us the information that
$d_{ij}^*$ is a semimetric, and nothing else. So, we need to be able
to somehow handle arbitrary metric spaces---precisely what
Bourgain's theorem does, by furnishing $f_1,\ldots,f_n\in L_1$ such
that for all $i,j\in \{1,\ldots,n\}$ we have
\begin{equation}\label{eq:use bourgain}
\frac{d_{ij}^*}{\log n}\lesssim \|f_i-f_j\|_1\le d_{ij}^*.
\end{equation}
Now,
\begin{multline}\label{eq:deduce from 14}
 \Phi^*(C,D)\stackrel{\eqref{reformulate Phi*}}{\le}
\frac{\sum_{i=1}^n\sum_{j=1}^n C(i,j) \|f_i-f_j\|_1}
               {\sum_{i=1}^n\sum_{j=1}^n D(i,j)\|f_i-f_j\|_1}
              \\ \stackrel{\eqref{eq:use bourgain}}{\lesssim} \log n\cdot \frac{\sum_{i=1}^n\sum_{j=1}^n C(i,j) d_{ij}^*}
               {\sum_{i=1}^n\sum_{j=1}^n D(i,j)d_{ij}^*}=\log
               n\cdot M^*(C,D).
\end{multline}
Thus, $\frac{\Phi^*(C,D)}{\log n}\lesssim M^*(C,D)\le \Phi^*(C,D)$,
i.e., the polynomial time algorithm of computing $M^*(C,D)$ is
guranteed to produce a number which is within a factor $\lesssim
\log n$ of $\Phi^*(C,D)$.

\begin{rem}\label{rem:the cut}
In the above argument we only discussed the algorithmic task of fast
estimation of the number $\Phi^*(C,D)$, rather than the problem of
producing in polynomial time a subset $\emptyset\neq S\subsetneq
\{1,\ldots,n\}$ for which $\Phi^*(S)$ is close up to a certain
multiplicative guarantee to the optimum value $\Phi^*(C,D)$. All the
algorithms discussed in this paper produce such a set $S$, rather than
just approximating the number $\Phi^*(C,D)$. In order to modify the
argument above to this setting, one needs to go into the proof of
Bourgain's embedding theorem, which as currently stated as just an existential
result for $f_1,\ldots,f_n$ as in~\eqref{eq:use bourgain}. This
issue is addressed in~\cite{LLR95}, which provides an algorithmic
version of Bourgain's theorem. Ensuing algorithms in this paper can
be similarly modified to produce a good cut $S$, but we will ignore
this issue from now on, and continue to focus solely on algorithms
for approximate computation of $\Phi^*(C,D)$.
\end{rem}

\subsection{The semidefinite program}\label{sec:semidefinite}

We have already stated in Section~\ref{SEC:EMBED} that the
logarithmic loss in the application~\eqref{eq:use bourgain} of
Bourgain's theorem cannot be improved. Thus, in order to obtain a
polynomial time algorithm with  approximation guarantee better than
$\lesssim \log n$, we need to impose additional geometric
restrictions on the metric $d_{ij}^*$; conditions that will
hopefully yield a class of metric spaces for which one can prove an
$L_1$ distortion bound that is asymptotically smaller than the
$\lesssim \log n$ of Bourgain's embedding theorem. This is indeed
possible, based on a quadratic variant of the discussion in
Section~\ref{sec:linear}; an approach due to Goemans and
Linial~\cite{Goe97,Lin-open,Lin02}.

The idea of Goemans and Linial is based on Corollary~\ref{cor:NEG},
i.e., on the fact that the metric space $L_1$ is of negative type.
We define $M^{**}(C,D)$ to be the minimum of
$\sum_{i=1}^n\sum_{j=1}^n C(i,j)d_{ij}$, subject to the constraint
that $\sum_{i=1}^n\sum_{j=1}^n D(i,j)d_{ij}=1$ and $d_{ij}$ is a
semimetric of negative type on $\{1,\ldots,n\}$. The latter
condition can be equivalently restated as the requirement that, in
addition to $d_{ij}$ being a semimetric on $\{1,\ldots,n\}$, there
exist vectors $v_1,\ldots,v_n\in L_2$ such that
$d_{ij}=\|v_i-v_j\|_2^2$ for all $i,j\in \{1,\ldots,n\}$.
Equivalently, there exists a symmetric positive semidefinite
$n\times n$ matrix $(a_{ij})$ (the Gram matrix of $v_1,\ldots,v_n)$,
such that $d_{ij}=a_{ii}+a_{jj}-2a_{ij}$ for all $i,j\in
\{1,\ldots,n\}$.

Thus, $M^{**}(C,D)$ is the minimum of $\sum_{i=1}^n\sum_{j=1}^n
C(i,j)(a_{ii}+a_{jj}-2a_{ij})$, a linear function in the variables
$(a_{ij})$, subject to the constraint that $(a_{ij})$ is a symmetric
positive semidefinite matrix, in conjunction with the linear
constraints $\sum_{i=1}^n\sum_{j=1}^n
D(i,j)(a_{ii}+a_{jj}-2a_{ij})=1$ and  for all $i,j,k\in
\{1,\ldots,n\}$, the triangle inequality constraint
$a_{ii}+a_{jj}-2a_{ij}\le
(a_{ii}+a_{kk}-2a_{ik})+(a_{kk}+a_{jj}-2a_{kj})$. Such an
optimization problem is called a semidefinite program, and by the
methods described in~\cite{GLS93}, $M^{**}(C,D)$ can be computed
with arbitrarily good precision in polynomial time.

Corollary~\ref{cor:NEG} and Lemma~\ref{lem:l1 enters} imply that
$M^*(C,D)\le M^{**}(C,D)\le \Phi^*(C,D)$. The following breakthrough
result of Arora, Rao and Vazirani~\cite{ARV04} shows that for
Sparsest Cut with uniform demands the Goemans-Linial approach does
indeed yield an improved approximation algorithm:

\begin{theorem}[\cite{ARV04}]\label{thm:ARV} In the case of uniform
demands, i.e., if $C(i,j)\in \{0,1\}$ and $D(i,j)=1$ for all $i,j\in
\{1,\ldots,n\}$, we have
\begin{equation}\label{eq:ARV}
\frac{\Phi^*(C,D)}{\sqrt{\log n}}\lesssim M^{**}(C,D)\le
\Phi^*(C,D).
\end{equation}
\end{theorem}
In the case of general demands we have almost the same result, up to
lower order factors:
\begin{theorem}[\cite{ALN08}]\label{thm:aln}
For all symmetric $C,D:\{1,\ldots,n\}\times \{1,\ldots,n\}\to
[0,\infty)$ we have
\begin{equation}\label{eq:aln}
\frac{\Phi^*(C,D)}{(\log n)^{\frac12+o(1)}}\lesssim M^{**}(C,D)\le
\Phi^*(C,D).
\end{equation}
\end{theorem}
The $o(1)$ term in~\eqref{eq:aln} is $\lesssim \frac{\log\log\log
n}{\log\log n}$. We conjecture that it could be removed altogether,
though at present it seems to be an inherent artifact of
complications in the proof in~\cite{ALN08}.

Before explaining some of the ideas behind the proofs of
Theorem~\ref{thm:ARV} and Theorem~\ref{thm:aln} (the full details
are quite lengthy and are beyond the scope of this survey), we
prove, following~\cite[Prop.\ 15.5.2]{Mat02}, a crucial identity
(attributed in~\cite{Mat02} to Y. Rabinovich) which reformulates
these results in terms of an $L_1$ embeddability problem.

\begin{lemma}\label{duality}
We have
\begin{multline}\label{eq;duality}
\sup\left\{\frac{\Phi^*(C,D)}{M^{**}(C,D)}:\
C,D:\{1,\ldots,n\}\times \{1,\ldots,n\}\to (0,\infty)\right\}\\=
\sup\Big\{c_1\big(\{1,\ldots,n\},d\big):\ d\ \mathrm{is\  a\ metric\
of\  negative\  type}\Big\}.
\end{multline}
\end{lemma}
\begin{proof} The proof of the fact that the left hand side
of~\eqref{eq;duality} is at most the right hand side
of~\eqref{eq;duality} is identical to the way~\eqref{eq:deduce from
14} was deduced from~\eqref{eq:use bourgain}.

In the reverse direction, let $d^*$ be a metric of negative type on
$\{1,\ldots,n\}$ for which $c_1(\{1,\ldots,n\},d^*)\eqdef c$ is
maximal among all such metrics. Let $\mathscr
C\subseteq \R^{n^2}$ be the cone in the space of $n\times n$
symmetric matrices from the last paragraph of Section~\ref{sec:L1},
i.e., $\mathscr C$ consists of all matrices of the form
$(\|f_i-f_j\|_1)$ for some $f_1,\ldots,f_n\in L_1$.

Fix $\e\in (0,c-1)$ and let $\mathscr K_\e\subseteq \R^{n^2}$ be the
set of all symmetric matrices $(a_{ij})$ for which there exists
$s>0$ such that $sd^*(i,j)\le a_{ij}\le (c-\e)sd^*(i,j)$ for all
$i,j\in \{1,\ldots,n\}$. By the definition of $c$, the convex sets
$\mathscr C$ and $\mathscr K_\e$ are disjoint, since otherwise $d^*$
would admit an embedding into $L_1$ with distortion $c-\e$. It
follows that there exists a symmetric matrix $(h_{ij}^\e)\in
\R^{n^2}\setminus\{0\}$ and $\alpha\in \R$, such that
$\sum_{i=1}^n\sum_{j=1}^n h_{ij}^\e a_{ij} \le \alpha$ for all
$(a_{ij})\in \mathscr K_\e$, and $\sum_{i=1}^n\sum_{j=1}^n h_{ij}^\e
b_{ij} \ge \alpha$ for all $(b_{ij})\in \mathscr C$. Since both
$\mathscr C$ and $\mathscr K_\e$ are closed under multiplication by
positive scalars, necessarily $\alpha=0$.

Define $C^\e(i,j)\eqdef h_{ij}^\e\1_{\{h_{ij}^\e\ge 0\}}$ and
$D^\e(i,j)\eqdef|h_{ij}^\e|\1_{\{h_{ij}^\e\le 0\}}$. By definition of
$M^{**}(C^\e,D^\e)$,
\begin{equation}\label{eq:use separation1}
\sum_{i=1}^n\sum_{j=1}^n C^\e(i,j)d^*_{ij}\ge
M^{**}(C^\e,D^\e)\cdot\sum_{i=1}^n\sum_{j=1}^n D^\e(i,j)d^*_{ij}.
\end{equation}
By considering $a_{ij}\eqdef\left((c-\e)\1_{\{h_{ij}^\e\ge
0\}}+\1_{\{h_{ij}^\e< 0\}}\right)d^*(i,j)\in \mathscr K_\e$, the
inequality $\sum_{i=1}^n\sum_{j=1}^n h_{ij}^\e a_{ij} \le 0$
becomes:
\begin{equation}\label{eq:use separation2}
\sum_{i=1}^n\sum_{j=1}^n D^\e(i,j)d^*_{ij}\ge
(c-\e)\sum_{i=1}^n\sum_{j=1}^n C^\e(i,j)d^*_{ij}.
\end{equation}
A combination of~\eqref{eq:use separation1} and~\eqref{eq:use
separation2} implies that $(c-\e)M^{**}(C^\e,D^\e)\le 1$. At the
same time, for all $f_1,\ldots,f_n\in L_1$, the inequality
$\sum_{i=1}^n\sum_{j=1}^n h_{ij}^\e\|f_i-f_j\|_1\ge 0$ is the same
as $\sum_{i=1}^n\sum_{j=1}^n C^\e(i,j)\|f_i-f_j\|_1\ge
\sum_{i=1}^n\sum_{j=1}^n D^\e(i,j)\|f_i-f_j\|_1$, which by
Lemma~\ref{reformulate Phi*} means that $\Phi^*(C^\e,D^\e)\ge 1$.
Thus $\Phi^*(C^\e,D^\e)/M^{**}(C^\e,D^\e)\ge c-\e$, and since this
holds for all $\e\in (0,c-1)$, the proof of Lemma~\ref{duality} is
complete.
\end{proof}

In the case of Sparsest Cut with uniform demands, we have the
following result which is analogous to Lemma~\ref{duality}, where the
$L_1$ bi-Lipschitz distortion is replaced by the smallest possible
factor by which $1$-Lipschitz functions into $L_1$ can distort the
{\em average distance}. The proof is a slight variant of the proof
of Lemma~\ref{duality}; the simple details are left to the reader.
This connection between Sparsest Cut with uniform demands and
embeddings that preserve the average distance is due to
Rabinovich~\cite{Rab08}.

\begin{lemma}\label{lem:average distance}
The supremum of $\Phi^*(C,D)/M^{**}(C,D)$ over all instances of
uniform demands, i.e., when $C(i,j)\in \{0,1\}$ and $D(i,j)=1$ for
all $i,j\in \{1,\ldots,n\}$, equals the infimum over $A>0$ such that
for all metrics $d$ on $\{1,\ldots,n\}$ of negative type, there
exist $f_1,\ldots,f_n\in L_1$ satisfying $\|f_i-f_j\|_1\le d(i,j)$
for all $i,j\in \{1,\ldots,n\}$ and
$A\sum_{i=1}^n\sum_{j=1}^n\|f_i-f_j\|_1\ge
\sum_{i=1}^n\sum_{j=1}^nd(i,j)$.
\end{lemma}

\subsubsection{$L_2$ embeddings of negative type metrics} The proof of Theorem~\ref{thm:ARV} in~\cite{ARV04} is based
on a clever geometric partitioning procedure for metrics of negative
type. Building heavily on ideas of~\cite{ARV04}, in conjunction with
some substantial additional combinatorial arguments, an alternative
approach to Theorem~\ref{thm:ARV} was obtained in~\cite{NRS05},
based on a purely graph theoretical statement which is of
independent interest. We shall now sketch this approach, since it is
modular and general, and as such it is useful for additional
geometric corollaries. 
We refer to~\cite{NRS05} for more information on these
additional applications, as well as to~\cite{ARV04} for the original
proof of Theorem~\ref{thm:ARV}.

Let $G=(V,E)$ be an $n$-vertex graph. The {\em vertex expansion} of $G$, denoted $h(G)$, is the largest $h\ge 0$ such that every $S\subseteq V$ with $|S|\le n/2$ has at least $h|S|$ neighbors in $V\setminus S$. The {\em edge expansion} of $G$, denoted $\alpha(G)$, is the largest  $\alpha\ge 0$ such that for every $S\subseteq V$ with $|S|\le n/2$, the number of edges joining $S$ and $V\setminus S$ is at least $\alpha|S|\cdot\frac{|E|}{n}$. The main combinatorial statement of~\cite{NRS05} relates these two notions of expansion of graphs:

\begin{theorem}[Edge Replacement Theorem~\cite{NRS05}]\label{thm:replace}
For every graph $G=(V,E)$ with $h(G)\ge \frac12$ there is a set of
edges $E'$ on $V$ with $\alpha(V,E')\gtrsim 1$, and such that for
every $uv\in E'$ we have $d_G(u,v)\lesssim \sqrt{\log |V|}$. Here
$d_G$ is the shortest path metric on $G$ (with respect to the
original edge set $E$), and all implicit constants are universal.
\end{theorem}

It is shown in~\cite{NRS05} that the $\lesssim \sqrt{\log n}$ bound on the length of the new edges in Theorem~\ref{thm:replace} is asymptotically tight. The proof of Theorem~\ref{thm:replace} is involved, and cannot be described here: it has two components, a combinatorial construction, as well a purely Hilbertian geometric argument based on, and simpler than, the original algorithm of~\cite{ARV04}. We shall now explain how Theorem~\ref{thm:replace} implies Theorem~\ref{thm:ARV} (this is somewhat different from the deduction in~\cite{NRS05}, which deals with a different semidefinite program for Sparsest Cut with uniform demands).

\begin{proof}[Proof of Theorem~\ref{thm:ARV} assuming Theorem~\ref{thm:replace}]
An application of (the easy direction of) Lemma~\ref{lem:average distance} shows that in order to prove Theorem~\ref{thm:ARV} it suffices to show that if $(\M,d)$ is an $n$-point metric space of negative type, with $\frac{1}{n^2}\sum_{x,y\in \M}d(x,y)=1$, then there exists a mapping $F:\M\to \R$ which is $1$-Lipschitz and such that $\frac{1}{n^2}\sum_{x,y\in \M}|F(x)-F(y)|\gtrsim 1/\sqrt{\log n}$. In what follows we use the standard notation for closed balls: for $x\in \M$ and $t\ge 0$, set $B(x,t)=\{y\in \M:\ d(x,y)\le t\}$.

Choose $x_0\in \M$ with $\frac{1}{n}\sum_{y\in \M} d(x_0,y)=r\eqdef
\min_{x\in \M}\frac{1}{n}\sum_{y\in \M} d(x,y)$.  Then $r\le
\frac{1}{n^2}\sum_{x,y\in \M} d(x,y)=1$, implying $1\ge
\frac{1}{n}\sum_{y\in \M} d(x_0,y)> \frac{2}{n}|\M\setminus
B(x_0,2)|$, or $|B(x_0,2)|> n/2$. Similarly $|B(x_0,4)|> 3n/4$.

Assume first that $\frac{1}{n^2}\sum_{x,y\in B(x_0,4)}d(x,y)\le
\frac14$ (this will be the easy case). Then
\begin{multline*}
1=\frac{1}{n^2}\sum_{x,y\in \M} d(x,y)\le \frac14+\frac{2}{n^2}\sum_{x\in \M}\sum_{y\in \M\setminus B(x_0,4)}\Big(d(x,x_0)+d(x_0,y)\Big)\\=\frac14+\frac{2r}{n}|\M\setminus B(x_0,4)|+\frac{2}{n}\sum_{y\in \M\setminus B(x_0,4)}d(x_0,y)\le \frac34+\frac{2}{n}\sum_{y\in \M\setminus B(x_0,4)}d(x_0,y),
\end{multline*}
or $\frac{1}{n}\sum_{y\in \M\setminus B(x_0,4)}d(x_0,y)\ge \frac18$.
Define a $1$-Lipschitz mapping $F:\M\to \R$ by
$F(x)=d\big(x,B(x_0,2)\big)=\min_{y\in B(x_0,2)} d(x,y)$.  The
triangle inequality implies that for every $y\in \M\setminus
B(x_0,4)$ we have $F(y)\ge \frac12d(y,x_0)$. Thus
\begin{multline*}
\frac{1}{n^2}\sum_{x,y\in \M} |F(x)-F(y)|\ge \frac{|B(x_0,2)|}{n^2}\sum_{y\in \M\setminus B(x_0,4)}d\big(y,B(x_0,2)\big)\\>\frac{1}{2n}\sum_{y\in \M\setminus B(x_0,4)}\frac12d(y,x_0)\gtrsim 1=\frac{1}{n^2}\sum_{x,y\in \M} d(x,y).
\end{multline*}
This completes the easy case, where there is even no loss of
$1/\sqrt{\log n}$ (and we did not use yet the assumption that $d$ is
a metric of negative type).

We may therefore assume from now on that $\frac{1}{n^2}\sum_{x,y\in
B(x_0,4)}d(x,y)\ge\frac14$.
The fact that $d$ is of negative type means that there are vectors $\{v_x\}_{x\in \M}\subseteq L_2$ such that $d(x,y)=\|v_x-v_y\|_2^2$ for all $x,y\in \M$. 

We will show that for a small enough universal constant $\e>0$,
there are two sets $S_1,S_2\subseteq B(x_0,4)$ such that
$|S_1|,|S_2|\ge\e n$ and $d(S_1,S_2)\ge \e^2/\sqrt{\log n}$. Once
this is achieved, the mapping $F:\M\to \R$ given by $F(x)=d(x,S_1)$
will satisfy $\frac{1}{n^2}\sum_{x,y\in \M} |F(x)-F(y)|\ge
\frac{2}{n^2}|S_1|\cdot |S_2|\frac{\e^2}{\sqrt{\log n}}\ge
\frac{2\e^4}{\sqrt{\log n}}$, as desired.

Assume for contradiction that no such $S_1,S_2$ exist. Define a set
of edges $E_0$ on $B(x_0,4)$ by $E_0\eqdef \Big\{\{x,y\}\subseteq
B(x_0,4):\ x\neq y\ \wedge \ d(x,y)<\e^2/\sqrt{\log n}\Big\}$. Our
contrapositive assumption says that any two subsets
$S_1,S_2\subseteq B(x_0,4)$ with $|S_1|,|S_2|\ge \e n\ge
\e|B(x_0,4)|$ are joined by an edge from $E_0$. By a (simple)
general graph theoretical lemma (see~\cite[Lem\ 2.3]{NRS05}), this
implies that, provided $\e\le 1/10$, there exists a subset
$V\subseteq B(x_0,4)$ with $|V|\ge (1-\e)|B(x_0,4)|\gtrsim n$, such
that the graph induced by $E_0$ on $V$, i.e., $G=\left(V,E=E_0\cap
{V\choose 2}\right)$, has  $h(G)\ge \frac12$.

We are now in position to apply the Edge Replacement Theorem, i.e., Theorem~\ref{thm:replace}. We obtain a new set of edges $E'$ on $V$ such that $\alpha(V,E')\gtrsim 1$ and for every $xy\in E'$ we have $d_G(x,y)\lesssim \sqrt{\log n}$. The latter condition means that there exists a path $\{x=x_0,x_1,\ldots,x_m=y\}\subseteq V$ such that $m\lesssim \sqrt{\log n}$ and $x_{i}x_{i-1}\in E$ for every $i\in \{1,\ldots,m\}$. By the definition of $E$, this implies that
\begin{equation}\label{eq:chain}
xy\in E'\implies d(x,y)\le \sum_{i=1}^nd(x_i,x_{i-1})\le m\frac{\e^2}{\sqrt{\log n}}\lesssim \e^2.
\end{equation}

It is a standard fact (the equivalence between edge expansion and a Cheeger inequality) that for every $f:V\to L_1$ we have
\begin{equation}\label{eq:cheeger}
\frac{1}{|E'|}\sum_{xy\in E'} \|f(x)-f(y)\|_1\ge \frac{\alpha(V,E')}{2|V|^2}\sum_{x,y\in V}\|f(x)-f(y)\|_1.
\end{equation}
For a proof of~\eqref{eq:cheeger} see~\cite[Fact\ 2.1]{NRS05}: this is a simple consequence of the cut cone representation, i.e., Corollary~\ref{cor:cut}, since the identity~\eqref{eq:superposition} shows that it suffices to prove~\eqref{eq:cheeger} when $f(x)=\1_S(x)$ for some $S\subseteq V$, in which case the desired inequality follows immediately from the definition of the edge expansion $\alpha(V,E')$.

Since $L_2$ is isometric to a subset of $L_1$ (see, e.g., \cite{Woj91}), it follows from~\eqref{eq:cheeger} and the fact that $\alpha(V,E')\gtrsim 1$ that
\begin{multline}\label{eq:count}
\e\stackrel{\eqref{eq:chain}}{\gtrsim} \frac{1}{|E'|}\sum_{xy\in E'}\sqrt{d(x,y)}=\frac{1}{|E'|}\sum_{xy\in E'}\|v_x-v_y\|_2\\\gtrsim \frac{1}{|V|^2}\sum_{x,y\in V}\|v_x-v_y\|_2\gtrsim \frac{1}{n^2}\sum_{x,y\in V}\sqrt{d(x,y)}.
\end{multline}
Now comes the point where we use the assumption $\frac{1}{n^2}\sum_{x,y\in B(x_0,4)}d(x,y)\ge\frac14$. Since for any $x,y\in B(x_0,4)$ we have $d(x,y)\le 8$, it follows that the number of pairs $(x,y)\in B(x_0,4)\times B(x_0,4)$ with $d(x,y)\ge 1/8$ is at least $n^2/64$. Since $|V|\ge (1-\e)|B(x_0,4)|$, the number of such pairs which are also in $V\times V$ is at least $\frac{n^2}{64}-3\e n^2\gtrsim n^2$, provided $\e$ is small enough. Thus $\frac{1}{n^2}\sum_{x,y\in V}\sqrt{d(x,y)}\gtrsim 1$, and~\eqref{eq:count} becomes a contradiction for small enough $\e$.
\end{proof}

\begin{rem}
The above proof of Theorem~\ref{thm:replace} used very little of the
fact that $d$ is a metric of negative type. In fact, all that was
required was that $d$ admits a quasisymmetric embedding into $L_2$;
see~\cite{NRS05}.
\end{rem}

It remains to say a few words about the proof of
Theorem~\ref{thm:aln}. Unfortunately, the present proof of this
theorem is long and involved, and it relies on a variety of results
from metric embedding theory. It would be of interest to obtain a
simpler proof. Lemma~\ref{duality} implies that
Theorem~\ref{thm:aln} is a consequence of the following embedding
result:

\begin{theorem}[\cite{ALN08}]\label{thm:embed ALN}
Every $n$-point metric space of negative type embeds into Hilbert
space with distortion $\lesssim (\log n)^{\frac12+o(1)}$.
\end{theorem}
Theorem~\ref{thm:embed ALN} improves over the previously
known~\cite{CGR08} bound of $\lesssim (\log n)^{3/4}$ on the
Euclidean distortion of $n$-point metric spaces of negative type. As
we shall explain below, Theorem~\ref{thm:embed ALN} is tight up to the $o(1)$ term.

The proof of Theorem~\ref{thm:embed ALN} uses the following notion
from~\cite{ALN08}:
\begin{definition}[Random zero-sets~\cite{ALN08}]\label{def:zero} Fix $\Delta,\ \zeta>0$,  and $p\in (0,1)$. A metric space $(\M,d)$
is said to admit a random zero set at scale $\Delta$, which is
$\zeta$-spreading with probability $p$, if there is a probability
distribution $\mu$ over subsets $Z\subseteq \M$ such that  $\mu
\left(\left\{Z:\ y\in Z\ \wedge\ d(x,Z)\ge
\Delta/\zeta\right\}\right)\ge p$ for every $x,y\in \M$ with
$d(x,y)\ge \Delta$. We denote by $\zeta(\M;p)$ the least $\zeta>0$
such that for every $\Delta>0$, $\M$ admits a random zero set at
scale $\Delta$ which is $\zeta$-spreading with probability $p$.
\end{definition}
The connection to metrics of negative type is due to the following
theorem, which can be viewed as the main structural consequence
of~\cite{ARV04}. Its proof uses~\cite{ARV04} in conjunction with two
additional ingredients: an analysis of the algorithm of~\cite{ARV04}
due to~\cite{Lee05}, and a clever iterative application of the
algorithm of~\cite{ARV04}, due to~\cite{CGR08}, while carefully
reweighting points at each step.

\begin{theorem}[Random zero sets for negative type
metrics]\label{thm:arv zero} There exists a universal constant $p>0$
such that any $n$-point metric space $(\M,d)$ of negative type
satisfies $\zeta(\M;p)\lesssim \sqrt{\log n}$.
\end{theorem}

Random zero sets are related to embeddings as follows. Fix
$\Delta>0$. Let $(\M,d)$ be a finite metric space, and fix
$S\subseteq \M$. By the definition of $\zeta(S;p)$, there exists a
distribution $\mu$ over subsets $Z\subseteq S$ such that for every
$x,y\in S$ with $d(x,y)\ge \Delta$ we have $ \mu
\left(\left\{Z\subseteq S:\ y\in Z\ \wedge\ d(x,Z)\ge
\Delta/\zeta(S;p)\right\}\right)\ge p$. Define $\f_{S,\Delta}:\M\to
L_2(\mu)$ by $\f_{S,\Delta}(x)=d(x,Z)$. Then $\f_{S,\Delta}$ is
$1$-Lipschitz, and for every $x,y\in S$ with $d(x,y)\ge \Delta$,
\begin{eqnarray}\label{eq:single}
\left\|\f_{S,\Delta}(x)-\f_{S,\Delta}(y)\right\|_{L_2(\mu)}=
\left(\int_{2^S} \left[d(x,Z)-d(y,Z)\right]^2d\mu(Z)\right)^{1/2}
\ge \frac{\Delta\sqrt{p}}{\zeta(S;p)}.
\end{eqnarray}

The remaining task is to ``glue" the mappings $\{\f_{S,\Delta}:\
\Delta>0,\ S\subseteq \M\}$ to form an embedding of $\M$ into
Hilbert space with the distortion claimed in
Theorem~\ref{thm:embed ALN}. A key ingredient of the proof of
Theorem~\ref{thm:embed ALN} is the embedding method called
``Measured Descent", that was developed in~\cite{KLMN05}. The
results of~\cite{KLMN05} were stated as embedding theorems rather than a gluing procedure; the
realization that a part of the arguments of~\cite{KLMN05} can be formulated explicitly as a general ``gluing lemma" is due
to~\cite{Lee05}. In~\cite{ALN08} it was necessary to enhance
the Measured Descent technique in order to prove the following key
theorem, which together with~\eqref{eq:single} and
Theorem~\ref{thm:arv zero} implies Theorem~\ref{thm:embed ALN}. See also~\cite{ALN-frechet} for a different enhancement of Measured Descent, which also implies Theorem~\ref{thm:embed ALN}. The
proof of Theorem~\ref{thm:new gluing} is quite intricate; we refer
to~\cite{ALN08} for the details.

\begin{theorem}\label{thm:new gluing}
Let $(\M,d)$ be an $n$-point metric space. Suppose that there is
 $\e\in [1/2,1]$ such
that for every $\Delta > 0$, and every subset $S \subseteq \M$,
there exists a 1-Lipschitz map $\f_{S,\Delta} : \M \to L_2$ with $
||\f_{S,\Delta}(x) - \f_{S,\Delta}(y)||_2 \gtrsim \Delta/(\log
|S|)^\varepsilon$ whenever $x,y \in S$ and $d(x,y) \ge \Delta$. Then
$c_2(\M) \lesssim(\log n)^{\varepsilon} \log \log n$.
\end{theorem}

The following corollary is an obvious consequence of
Theorem~\ref{thm:embed ALN}, due to the fact that $L_1$ is a metric
space of negative type.
\begin{cor}\label{cor:L1}
Every $X\subseteq L_1$ embeds into $L_2$ with distortion $\lesssim
(\log |X|)^{\frac12+o(1)}$.
\end{cor}
We stated Corollary~\ref{cor:L1} since it is of special importance:
in 1969, Enflo~\cite{Enflo69} proved that the Hamming cube, i.e.,
$\{0,1\}^k$ equipped with the metric induced  from $\ell_1^k$, has
Euclidean distortion $\sqrt{k}$. Corollary~\ref{cor:L1} says that up
to lower order factors, the Hamming cube is among the most non-Euclidean
subset of $L_1$. There are very few known results of this type,
i.e., (almost) sharp evaluations of the largest Euclidean distortion
of an $n$-point subset of a natural metric space. A notable such
result is Matou\v{s}ek's theorem~\cite{Mat99} that any $n$-point
subset of the infinite binary tree has Euclidean distortion
$\lesssim \sqrt{\log \log n}$, and consequently, due to~\cite{BS05},
the same holds true for $n$-point subsets of, say, the hyperbolic
plane. This is tight due to Bourgain's matching lower
bound~\cite{Bou-tree} for the Euclidean distortion of finite depth
complete binary trees.

\subsubsection{The Goemans-Linial conjecture}\label{sec:GL} Theorem~\ref{thm:aln}
is the best known approximation algorithm for the Sparsest Cut
Problem (and Theorem~\ref{thm:ARV} is the best known algorithm in
the case of uniform demands). But, a comparison of
Lemma~\ref{duality} and Theorem~\ref{thm:embed ALN} reveals a
possible avenue for further improvement: Theorem~\ref{thm:embed ALN}
produces an embedding of negative type metrics into $L_2$ (for which
the bound of Theorem~\ref{thm:embed ALN} is sharp up to lower order
factors), while for Lemma~\ref{duality} all we need is an embedding
into the larger space $L_1$. It was conjectured by Goemans and
Linial (see~\cite{Goe97,Lin-open,Lin02} and~\cite[pg.\
379--380]{Mat02}) that any finite metric space of negative type
embeds into $L_1$ with distortion $\lesssim 1$. If true, this would
yield, via the Goemans-Linial semidefinite relaxation, a constant
factor approximation algorithm for Sparsest Cut.

As we shall see below, it turns out that the Goemans-Linial
conjecture is false, and in fact there exist~\cite{CKN09}
arbitrarily large $n$-point metric spaces $\M_n$ of negative type
for which $c_1(\M_n)\ge (\log n)^c$, where $c$ is a universal
constant. Due to the duality argument in Lemma~\ref{duality}, this
means that the algorithm of Section~\ref{sec:semidefinite} is doomed
to make an error of at least $(\log n)^c$, i.e., there exist
capacity and demand functions $C_n,D_n:\{1,\ldots,n\}\times
\{1,\ldots,n\}\to [0,\infty)$ for which we have
$M^{**}(C_n,D_n)\lesssim \Phi^*(C_n,D_n)/(\log n)^c$. Such a
statement is referred to in the literature as the fact that the {\em
integrality gap} of the Goemans-Linial semidefinite relaxation of
Sparsest Cut is at least $(\log n)^c$.

\subsubsection{Unique Games hardness and the Khot-Vishnoi integrality
gap}\label{sec:UGC}\label{sec:unique} Khot's Unique Games
Conjecture~\cite{Khot02} is that for every $\e>0$ there exists a
prime $p=p(\e)$ such that there is no polynomial time algorithm
that, given $n\in \N$ and a system of $m$-linear equations in
$n$-variables of the form $x_i-x_j=c_{ij}\mod p$ for some $c_{ij}\in
\N$, determines whether there exists an assignment of an integer
value to each variable $x_i$ such that at least $(1-\e)m$ of the
equations are satisfied, or whether no assignment of such values can satisfy
more than $\e m$ of the equations (if neither of these possibilities occur, then an arbitrary output is allowed). This formulation of the
conjecture is due to~\cite{KKMO07}, where it is shown that it is equivalent to the
original formulation in~\cite{Khot02}. The Unique Games Conjecture
is by now a common assumption that has numerous applications in
computational complexity; see the survey~\cite{Khot10} (in this
collection) for more information.

In~\cite{KV04,CKKRS06} it was shown that the existence of a
polynomial time constant factor approximation algorithm for Sparsest Cut would refute the Unique
Games Conjecture, i.e., one can use a polynomial time constant factor approximation algorithm for
Sparsest Cut to solve in polynomial time the above algorithmic task
for linear equations.

For a period of time in 2004, this computational hardness result led to a strange situation: either the complexity theoretic Unique Games Conjecture is true, or the purely geometric Goemans-Linial conjecture is true, but not both. In a remarkable tour de force, Khot and Vishnoi~\cite{KV04} delved into the proof of their hardness result and managed to construct from it a concrete family of arbitrarily large $n$-point metric spaces $\M_n$ of negative type for which $c_1(\M_n)\gtrsim (\log \log n)^c$, where $c$ is a universal constant, thus refuting the Goemans-Linial conjecture. Subsequently, these Khot-Vishnoi metric spaces $\M_n$ were analyzed in~\cite{KR06}, resulting in the lower bound $c_1(\M_n)\gtrsim \log \log n$. Further work in~\cite{DKSV06} yielded a $\gtrsim \log\log n$ integrality gap for Sparsest Cut with uniform demands, i.e., ``average distortion"  $L_1$ embeddings (in the sense of Lemma~\ref{lem:average distance}) of negative type metrics   were ruled out as well.


\subsubsection{The Bretagnolle, Dacunha-Castelle, Krivine theorem and invariant metrics on Abelian groups}\label{sec:krivine}
A combination of Schoenberg's classical
characterization~\cite{schoenberg} of metric spaces that are
isometric to subsets of Hilbert space, and a theorem of Bretagnolle,
Dacunha-Castelle and Krivine~\cite{BDCK66} (see
also~\cite{WellWill}), implies that if $p\in [1,2]$ and
$(X,\|\cdot\|_X)$ is a separable Banach space such that the metric
space $(X,\|x-y\|_X^{p/2})$ is isometric to a subset of Hilbert
space, then $X$ is (linearly) isometric to a subspace of $L_p$.
Specializing to $p=1$ we see that the Goemans-Linial conjecture is
true for Banach spaces. With this motivation for the Goemans-Linial
conjecture in mind, one notices that the Goemans-Linial conjecture
is part of a natural one parameter family of conjectures which
attempt to extend the theorem Bretagnolle, Dacunha-Castelle and
Krivine to general metric spaces rather than Banach spaces: is it
true that for $p\in [1,2)$ any metric space $(\M,d)$ for which
$(\M,d^{p/2})$ is isometric to a subset of $L_2$ admits a
bi-Lipschitz embedding into $L_p$? This generalized Goemans-Linial
conjecture turns out to be false for all $p\in [1,2)$; our example
based on the Heisenberg group furnishes counter-examples for all
$p$.

It is also known that certain invariant metrics on Abelian groups satisfy the Goemans-Linial conjecture:
\begin{theorem}[\cite{ANV07}]\label{thm:lamp}
Let $G$ be a finite Abelian group,
equipped with an invariant metric $\rho$. Suppose that $2\le m\in \mathbb
N$ satisfies $mx=0$ for all $x\in G$. Denote
$D=c_2\left(G,\sqrt{\rho}\right)$. Then $c_1(G,\rho)\lesssim
D^4\log m$.
\end{theorem}
It is an interesting open question whether the dependence on the
exponent $m$ of the group $G$ in Theorem~\ref{thm:lamp} is
necessary. Can one construct a counter-example to the Goemans-Linial
conjecture which is an invariant metric on the cyclic group
$C_n$ of order $n$? Or, is there for every $D\ge 1$ a constant
$K(D)$ such that for every invariant metric $\rho$ on $C_n$ for
which $c_2\left(G,\sqrt{\rho}\right)\le D$ we have $c_1(G,\rho)\le
K(D)$?

One can view the above discussion as motivation for why one might
consider the Heisenberg group as a potential counter-example to the
Goemans-Linial conjecture. Assuming that we are interested in
invariant metrics on groups, we wish to depart from the setting of
Abelian groups or Banach spaces, and if at the same time we would
like our example to have some useful analytic properties (such as
invariance under rescaling and the availability of a group norm),
the Heisenberg group suggests itself as a natural candidate. This
plan is carried out in Section~\ref{sec:heisenberg}.
\section{Embeddings of the Heisenberg group}\label{sec:heisenberg}

The purpose of this section is to discuss Theorem~\ref{thm:ckn} and
Theorem~\ref{thm:rate} from the introduction. Before doing so, we
have an important item of unfinished business: relating the
Heisenberg group to the Sparsest Cut Problem. We will do this in
Section~\ref{sec:snowflake}, following~\cite{LN06}.

In preparation, we need to recall the Carnot-Carath\'eodory geometry of the continuous Heisenberg group $\H$, i.e., $\R^3$ equipped with the non-commutative product $(a,b,c)\cdot (a',b',c')=
(a+a',b+b',c+c'+ab'-ba')$. Due to lack of space, this will have to be a crash course, and we refer  to the relevant introductory sections of~\cite{ckn} for a more thorough discussion.

The identity element of $\H$ is $e=(0,0,0)$, and the inverse element
of $(a,b,c)\in \H$ is $(-a,-b,-c)$. The center of $\H$ is the
$z$-axis $\{0\}\times\{0\}\times \R$. For $g\in \H$ the {\em
horizontal plane} at $g$ is defined as
$\H_g=g(\R\times\R\times\{0\})$. An affine line $L\subseteq \H$ is
called a {\em horizontal line} if for some $g\in \H$ it passes
through $g$ and is contained in the affine plane $\H_g$. The
standard scalar product $\langle \cdot,\cdot\rangle$ on $\H_e$
naturally induces a scalar product $\langle \cdot,\cdot\rangle_g$ on
$\H_g$ by $\langle gx,gy\rangle_g=\langle x,y\rangle$. Consequently,
we can define the Carnot-Carath\'eodory metric $d^\H$ on $\H$ by
letting $d^\H(g,h)$ be the infimum of lengths of smooth curves
$\gamma:[0,1]\to \H$ such that $\gamma(0)=g$, $\gamma(1)=h$ and for
all $t\in [0,1]$ we have $\gamma'(t)\in H_{\gamma(t)}$ (and, the
length of $\gamma'(t)$ is computed with respect to the scalar
product $\langle \cdot,\cdot\rangle_{\gamma(t)}$). The {\em ball-box
principle} (see~\cite{gromov}) implies that
$d^\H\big((a,b,c),(a',b',c')\big)$ is bounded above and below by a
constant multiple of $|a-a'|+|b-b'|+\sqrt{|c-c'+ab'-ba'|}$.
Moreover, since the integer grid $\H(\Z)$ is a discrete cocompact
subgroup of $\H$, the word metric $d_W$ on $\H(\Z)$ is bi-Lipschitz
equivalent to the restriction of $d^\H$ to $\H(\Z)$ (see, e.g,
\cite{BBI}). For $\theta>0$ define the dilation operator
$\delta_\theta:\H\to \H$ by $\delta_\theta(a,b,c)=(\theta a ,\theta
b,\theta^2 c)$. Then for all $g,h\in \H$ we have
$d^\H(\delta_\theta(g),\delta_\theta(h))=\theta d^\H(g,h)$. The
Lebesgue measure $\L_3$ on $\R^3$ is a Haar measure of $\H$, and the
volume of a $d^\H$-ball of radius $r$ is proportional to $r^4$.

\subsection{Heisenberg metrics with isometric $L_p$ snowflakes}\label{sec:snowflake} For every
$(a,b,c)\in \H$ and  $p\in [1,2)$, define
\begin{eqnarray*}\label{eq:defM}
M_p(a,b,c)=
\sqrt[4]{(a^2+b^2)^2+4c^2}\cdot \left(\cos\left(\frac{p}{2}\arccos\left(\frac{a^2+b^2}{\sqrt{(a^2+b^2)^2+4c^2}}\right)\right)\right)^{1/p}.
\end{eqnarray*}
It was shown in~\cite{LN06} that $M_p$ is a {\em group norm} on
$\H$, i.e., for all $g,h\in \H$ and $\theta\ge 0$ we have
$M_p(gh)\le M_p(g)+M_p(h)$, $M_p(g^{-1})=M_p(g)$ and
$M_p(\delta_\theta(g))=\theta M_p(g)$. Thus $d_p(g,h)\eqdef
M_p(g^{-1}h)$ is a left-invariant metric on $\H$. The metric $d_p$
is bi-Lipschitz equivalent to $d^\H$ with distortion of order
$1/\sqrt{2-p}$ (see~\cite{LN06}). Moreover, it was shown
in~\cite{LN06} that $(\H,d_p^{p/2})$ admits an isometric embedding
into $L_2$. Thus, in particular, the metric space $(\H,d_1)$, which
bi-Lipschitz equivalent to $(\H,d^\H)$, is of negative type.

The fact that $(\H,d^\H)$ does not admit a bi-Lipschitz embedding
into $L_p$ for any $1\le p<\infty$ will show that the generalized
Goemans-Linial conjecture (see Section~\ref{sec:krivine}) is false.
In particular, $(\H,d_1)$, and hence by a standard rescaling
argument also $(\H(\Z),d_1)$, is a counter-example to the
Goemans-Linial conjecture. Note that it is crucial here that we are
dealing with the function space $L_p$ rather than the sequence space
$\ell_p$, in order to use a compactness argument to deduce from this
statement that there exist arbitrarily large $n$-point metric spaces
$(\M_n,d)$ such that $(\M_n,d^{p/2})$ is isometric to a subset of
$L_2$, yet $\lim_{n\to \infty} c_p(\M_n)=\infty$. The fact that this
statement follows from non-embeddability into $L_p$ is a consequence
of a well known ultrapower argument (see~\cite{Hei80}), yet for
$\ell_p$ this statement is false (e.g., $\ell_2$ does not admit a
bi-Lipschitz embedding into $\ell_p$, but all finite subsets of
$\ell_2$ embed isometrically into $\ell_p$). Unfortunately, this
issue creates substantial difficulties in the case of primary
interest $p=1$. In the reflexive range $p>1$, or for a separable
dual space such as $\ell_1$ ($=c_0^*$), the non-embeddability of
$\H$ follows from a natural extension of a classical result of
Pansu~\cite{Pan89}, as we explain in Section~\ref{sec:pansu}. This
approach fails badly when it comes to embeddings into $L_1$: for
this purpose a novel method of Cheeger and Kleiner~\cite{ckbv} is
needed, as described in Section~\ref{sec:CK}.


\subsection{Pansu differentiability}\label{sec:pansu}

Let $X$ be a Banach space and $f:\H\to X$. Following~\cite{Pan89},
$f$ is said to have a Pansu derivative at $x\in \H$ if for every
$y\in \H$ the limit $ D_f^x(y)\eqdef \lim_{\theta\to 0}
\big(f(x\delta_\theta(y))-f(x)\big)/\theta$ exists, and $D^x_f:\H\to
X$ is a group homomorphism, i.e., for all $y_1,y_2\in \H$ we have
$D_f^x(y_1y_2^{-1})=D_f^x(y_1)-D_f^y(y_2)$. Pansu
proved~\cite{Pan89} that every $f:\H\to \R^n$ which is Lipschitz in
the metric $d^\H$ is Pansu differentiable almost everywhere. It was
observed in~\cite{LN06,CK06-RNP} that this result holds true if the
target space $\R^n$ is replaced by any Banach space with the
Radon-Nikod\'ym property, in particular $X$ can be any reflexive
Banach space such as $L_p$ for $p\in (1,\infty)$, or a separable
dual Banach space such as $\ell_1$. As noted by Semmes~\cite{Sem96},
this implies that $\H$ does not admite a bi-Lipschitz embedding into
any Banach space $X$ with the Radon-Nikod\'ym property: a
bi-Lipschitz condition for $f$ implies that at a point $x\in \H$ of
Pansu differentiability, $D_f^x$ is also bi-Lipschitz, and in
particular a group isomorphism. But that's impossible since $\H$ is
non-commutative, unlike the additive group of $X$.

\subsection{Cheeger-Kleiner differentiability}\label{sec:CK}
Differentiability theorems fail badly when the target space is
$L_1$, even for functions defined on $\R$; consider Aronszajn's
example~\cite{aronszajn} of the ``moving indicator function"
$t\mapsto \1_{[0,t]}\in L_1$. For $L_1$-valued Lipschitz functions
on $\H$, Cheeger and Kleiner~\cite{ckbv,ckmetmon} developed an
alternative differentiation theory, which is sufficiently strong to
show that $\H$ does not admit a bi-Lipschitz embedding into $L_1$.
Roughly speaking, a differentiation theorem states that in the
infinitesimal limit, a Lipschitz mapping converges to a mapping that
belongs to a certain ``structured" subclass of mappings (e.g.,
linear mappings or group homomorphisms). The Cheeger-Kleiner theory
shows that, in a sense that will be made precise below, $L_1$-valued
Lipschitz functions on $\H$ are in the infinitesimal limit similar
to Aronszajn's moving indicator.

For an open subset $U\subseteq \H$ let $\Cut(U)$ denote the space of
(equivalences classes up to measure zero) of measurable subsets of
$U$. Let $f:U\to L_1$ be a Lipschitz function. An infinitary variant
of the cut-cone decomposition of Corollary~\ref{cor:cut}
(see~\cite{ckbv}) asserts that there exists a measure $\Sigma_f$ on
$\Cut(U)$, such that for all $x,y\in U$ we have
$\|f(x)-f(y)\|_1=\int_{\Cut(U)}|\1_E(x)-\1_E(y)|d\Sigma_f(E)$. The
measure $\Sigma_f$ is called the {\em cut measure} of $f$. The idea
of Cheeger and Kleiner is to detect the ``infinitesimal regularity"
of $f$ in terms of the infinitesimal behavior of the measure
$\Sigma_f$; more precisely, in terms of the shape of the sets $E$ in
the support of $\Sigma_f$, after passing to an infinitesimal limit.

\begin{theorem}[Cheeger-Kleiner differentiability
theorem~\cite{ckbv,ckmetmon}]\label{thm:CKBV} For almost every $x\in
U$ there exists a measure $\Sigma_f^x$ on $\Cut(\H)$  such that for
all $y,z\in \H$ we have
\begin{equation}\label{eq:CK limit}
\lim_{\theta\to 0}
\frac{\|f(x\delta_\theta(y))-f(x\delta_\theta(z))\|_1}{\theta}=\int_{\Cut(\H)}|\1_E(y)-\1_E(z)|d\Sigma_f^x(E).
\end{equation}
Moreover, the measure $\Sigma_f^x$ is supported on affine
half-spaces whose boundary is a vertical plane, i.e., a plane which
isn't of the form $\H_g$ for some $g\in \H$ (equivalently, an
inverse image, with respect to the orthogonal projection from $\R^3$
onto  $\R\times\R\times \{0\}$, of a line in $\R\times\R\times
\{0\}$).
\end{theorem}
Theorem~\ref{thm:CKBV} is incompatible with $f$ being bi-Lipschitz,
since the right hand side of~\eqref{eq:CK limit} vanishes when $y,z$
lie on the same coset of the center of $\H$, while if $f$ is
bi-Lipschitz the left hand side of~\eqref{eq:CK limit} is at least a
constant multiple of $d^\H(y,z)$.

\subsection{Compression bounds for $L_1$ embeddings of the Heisenberg
group}\label{sec:compression} Theorem~\ref{thm:ckn} and
Theorem~\ref{thm:rate} are both a consequence of the following
result from~\cite{ckn}:
 \begin{theorem}[Quantitative central
 collapse~\cite{ckn}]\label{thm:quant}
\label{tmain} There exists a universal constant $c\in (0,1)$ such
that for every $p\in \H$, every $1$-Lipschitz $f:B(p,1)\to L_1$, and
every $\e\in \left(0,\frac14\right)$, there exists $r\ge \e$ such
that with respect to Haar measure, for at least half
 of the points $x\in B(p,1/2)$, at least half of
the points $(x_1,x_2)\in B(x,r)\times B(x,r)$ which lie on the same
coset of the center satisfy:
\begin{equation*}
\label{compression} \|f(x_1)-f(x_2)\|_{1}\leq
\frac{d^\H(x_1,x_2)}{(\log(1/\e))^c}\, .
\end{equation*}
\end{theorem}
It isn't difficult to see that Theorem~\ref{thm:quant} implies
Theorem~\ref{thm:ckn} and Theorem~\ref{thm:rate}. For example, in
the setting of Theorem~\ref{thm:ckn} we are given a bi-Lipschitz
embedding $f:\{1,\ldots,n\}^3\to L_1$, and using either the general
extension theorem of~\cite{LN05} or a partition of unity argument,
we can extend $f$ to a Lipschitz (with respect to $d^\H$) mapping
$\bar f:[1,n]^3\to L_1$, whose Lipschitz constant is at most a
constant multiple of the Lipschitz constant of $f$.
Theorem~\ref{thm:quant} (after rescaling by $n$)  produces a pair of
points $y,z\in [1,n]^3$ of distance $\gtrsim \sqrt{n}$, whose
distance is contracted under $\bar f$ by $\gtrsim (\log n)^c$. By
rounding $y,z$ to their nearest integer points in
$\{1,\ldots,n\}^3$, we conclude that $f$ itself must have
bi-Lipschitz distortion $\gtrsim (\log n)^c$. The deduction of
Theorem~\ref{thm:rate} from Theorem~\ref{thm:quant} is just as
simple; see~\cite{ckn}.

Theorem~\ref{thm:quant} is a quantitative version of
Theorem~\ref{thm:CKBV}, in the sense it gives a definite lower bound
on the macroscopic scale at which a given amount of collapse of
cosets of the center, as exhibited by the differentiation
result~\eqref{eq:CK limit}, occurs. As explained in~\cite[Rem.\
2.1]{ckn}, one cannot hope in general to obtain rate bounds in differentiation results such as~\eqref{eq:CK limit}. Nevertheless,
there are situations where ``quantitative differentiation results"
have been successfully proved; important precursors of Theorem~\ref{thm:quant}
include the work of Bourgain~\cite{Bou87}, Jones~\cite{jones88},
Matou\v{s}ek~\cite{Mat99}, and Bates, Johnson, Lindenstrauss,
Preiss, Schechtman~\cite{BJLPS99}. Specifically, we should mention
that Bourgain~\cite{Bou87} obtained a lower bound on $\e>0$ such
that any embedding of an $\e$-net in a unit ball of an
$n$-dimensional normed space $X$ into a  normed space $Y$ has
roughly the same distortion as the distortion required to embed all
of $X$ into $Y$, and Matou\v{s}ek~\cite{Mat99}, in his study of
embeddings of trees into uniformly convex spaces, obtained
quantitative bounds on the scale at which ``metric differentiation"
is almost achieved, i.e., a scale at which discrete geodesics are
mapped by a Lipschitz function to ``almost geodesics". These earlier
results are in the spirit  of Theorem~\ref{thm:quant},
though the proof of Theorem~\ref{thm:quant} in~\cite{ckn} is
substantially more involved.

We shall now say a few words on the proof of
Theorem~\ref{thm:quant}; for lack of space this will have to be a
rough sketch, so we refer to~\cite{ckn} for more details, as well as
to the somewhat different presentation in~\cite{CKN09}. Cheeger and
Kleiner obtained two different proofs of Theorem~\ref{thm:CKBV}. The
first proof~\cite{ckbv} started with the important observation that
the fact that $f$ is Lipschitz forces the cut measure $\Sigma_f$ to
be supported on sets with additional regularity, namely sets of {\em
finite perimeter}. Moreover, there is a definite bound on the {\em
total perimeter}: $\int_{\Cut(U)}
\mathrm{PER}(E,B(p,1))d\Sigma_f(E)\lesssim 1$, where
$\mathrm{PER}(E,B(p,1))$ denotes the perimeter of $E$ in the ball
$B(p,1)$ (we refer to the book~\cite{luigetal}, and the detailed
explanation in~\cite{ckbv,ckn} for more information on these
notions). Theorem~\ref{thm:quant} is then proved in~\cite{ckbv} via
an appeal to results~\cite{italians1,italians2} on the infinitesimal
structure of sets of finite perimeter in $\H$. A different proof of
Theorem~\ref{thm:quant} was found in~\cite{ckmetmon}. It is based on
the notion of  {\em metric differentiation}, which is used
in~\cite{ckmetmon} to reduce the problem to mappings $f:\H\to L_1$
for which the cut measure is supported on {\em monotone sets}, i.e.,
sets $E\subseteq \H$ such that for every horizontal line $L$, up to
a set of measure zero, both $L\cap E$ and $L\cap (\H\setminus E)$
are either empty or subrays of $L$. A non-trivial classification of
monotone sets is then proved in~\cite{ckmetmon}: such sets are up to
measure zero half-spaces.

This second proof of Theorem~\ref{thm:quant} avoids completely the
use of perimeter bounds. Nevertheless, the starting point of the
proof of Theorem~\ref{thm:quant} can be viewed as a hybrid argument,
which incorporates both perimeter bounds, and a new classification
of {\em almost} monotone sets. The quantitative setting of
Theorem~\ref{thm:quant} leads to issues that do not have analogues
in the non-quantitative proofs (e.g., the approximate classification
results of ``almost" monotone sets in balls cannot be simply that
such sets are close to half-spaces in the entire ball;
see~\cite[Example\ 9.1]{ckn}).

In order to proceed we need to quantify the extent to which a set
$E\subseteq B(x,r)$ is monotone. For a horizontal line $L\subseteq
\H$ define the non-convexity ${\rm NC}_{B(x,r)}(E,L)$ of $(E,L)$ on
$B(x,r)$ as the infimum of $\int_{L\cap B(x,r)}\left|\1_I-\1_{E\cap
L\cap B_r(x)}\right|d\mathcal H_L^1$ over all sub-intervals
$I\subseteq L\cap B_r(x)$. Here $\mathcal H_L^1$ is the
$1$-dimensional Hausdorff measure on $L$ (induced from the metric
$d^\H$). The non-monotonicity of $(E,L)$ on $B(x,r)$ is defined to
be ${\rm NM}_{B(x,r)}(E,L)\eqdef {\rm NC}_{B(x,r)}(E,L)+{\rm
NC}_{B(x,r)}(\H\setminus E,L)$. The total non-monotonicity of $E$ on
$B(x,r)$ is defined as:
\begin{equation*}
{\rm NM}_{B(x,r)}(E)\eqdef
\frac{1}{r^4}\int_{\mathrm{lines}(B(x,r))}{\rm
NM}_{B(x,r)}(E,L)d\mathcal{N}(L),
\end{equation*}
where $\mathrm{lines}(U)$ denotes the set of horizontal lines in
$\H$ which intersect $U$, and $\mathcal N$ is the left invariant
measure on $\mathrm{lines} (\H)$, normalized so that the measure of
$\mathrm{lines}(B(e,1))$ is $1$.

The following stability result for monotone sets constitutes the
bulk of~\cite{ckn}:
\begin{theorem}\label{thm:stability}
There exists a universal constant $a>0$ such that if a measurable
set $E\subseteq B(x,r)$ satisfies ${\rm NM}_{B(x,r)}(E)\le \e^a$
then there exists a half-space $\mathcal P$ such that
$$
\frac{\mathscr{L}_3\left((E\cap B_{\e r}(x))\triangle \mathcal
P\right)}{\mathscr{L}_3(B_{\e r}(x))}<\e^{1/3}.
$$
\end{theorem}
Perimeter bounds are used in~\cite{ckn,CKN09} for two purposes. The
first is finding a controlled scale $r$ such that at most locations,
apart from a certain collection of cuts, the mass of $\Sigma_f$ is
supported on subsets which satisfy the assumption of
Theorem~\ref{thm:stability} (see~\cite[Sec.\ 9]{CKN09}). But, the
excluded cuts may have infinite measure with respect to
$\Sigma_f$. Nonetheless, using perimeter bounds once more, together
with the isoperimetric inequality in $\H$ (see~\cite{Pan82,CDF94}),
it is shown that their contribution to the metric is negligibly
small (see~\cite[Sec.\ 8]{CKN09}).

By Theorem~\ref{thm:stability}, it remains to deal with the
situation where all the cuts in the support of $\Sigma_f$ are close
to half-spaces: note that we are not claiming in
Theorem~\ref{thm:stability} that the half-space is vertical.
Nevertheless, a simple geometric argument shows that even in the
case of cut measures that are supported on general (almost)
half-spaces, the mapping $f$ must significantly distort some
distances. The key point here is that if the cut measure is actually
supported on half spaces, then it follows (after the fact)  that for
{\em every  affine} line $L$, if $x_1,x_2,x_3\in L$ and $x_2$ lies
between $x_1$ and $x_3$ then
$\|f(x_1)-f(x_3)\|_1=\|f(x_1)-f(x_2)\|_1+\|f(x_2)-f(x_3)\|_1$. But
if $L$ is vertical then $d^\H|_L$ is bi-Lipschitz to the {\em square
root} of the difference of the $z$-coordinates, and it is trivial to
verify that this metric on $L$ is not bi-Lipschitz equivalent to a
metric on $L$ satisfying this additivity condition. For the details
of (a quantitative version of) this final step of the argument
see~\cite[Sec.\ 10]{CKN09}.


\bibliographystyle{abbrv}
\bibliography{ICMbib}



\end{document}